\documentclass[12pt,reqno]{article}

\usepackage[usenames]{color}
\usepackage{amssymb}
\usepackage{amscd}

\usepackage[colorlinks=true,
linkcolor=webgreen,
filecolor=webbrown,
citecolor=webgreen]{hyperref}

\definecolor{webgreen}{rgb}{0,.5,0}
\definecolor{webbrown}{rgb}{.6,0,0}

\usepackage{color}
\usepackage{fullpage}
\usepackage{float}

\usepackage{graphics,amsmath,amssymb}
\usepackage{amsthm}
\usepackage{amsfonts}
\usepackage{latexsym}

\newcommand{\seqnum}[1]{\href{http://oeis.org/#1}{\underline{#1}}}

\begin{document}


\theoremstyle{plain}
\newtheorem{theorem}{Theorem}
\newtheorem{corollary}[theorem]{Corollary}
\newtheorem{lemma}[theorem]{Lemma}
\theoremstyle{remark}
\newtheorem{remark}[theorem]{Remark}

\newcommand{\lrf}[1]{\left\lfloor #1\right\rfloor}

\begin{center}
\vskip 1cm{\LARGE\bf Lucas-Euler Relations using Balancing and Lucas-Balancing Polynomials \\
\vskip .11in }

\vskip 1cm

{\large  Robert Frontczak\footnote{Statements and conclusions made in this article by R. Frontczak are entirely those of the author. They do not necessarily reflect the views of LBBW.
} \\
Landesbank Baden-W\"urttemberg (LBBW) \\ Stuttgart,  Germany \\
\href{mailto:robert.frontczak@lbbw.de}{\tt robert.frontczak@lbbw.de}

\vskip 0.2 in

Taras Goy  \\ 
Faculty of Mathematics and Computer Science\\
Vasyl Stefanyk Precarpathian National University\\
Ivano-Frankivsk, Ukraine\\
\href{mailto:taras.goy@pnu.edu.ua}{\tt taras.goy@pnu.edu.ua}}
\end{center}

\vskip .2 in

\begin{abstract}
We establish some new combinatorial identities involving Euler polynomials and balancing (Lucas-balancing) polynomials.
The derivations use elementary techniques and are based on functional equations for the respective generating functions. 
From these polynomial relations, we deduce interesting identities with Fibonacci and Lucas numbers, and Euler numbers.
The results must be regarded as companion results to some Fibonacci-Bernoulli identities, which we derived in our previous paper. 

\end{abstract}


\section{Motivation and preliminaries}

In 1975, Byrd \cite{Byrd} derived the following identity relating Lucas numbers to Euler numbers:
\begin{equation} \label{byrd1}
\sum_{k=0}^{\lfloor{n/2}\rfloor} \binom{n}{2k} \left(\frac{5}{4}\right)^{k} L_{n-2k} E_{2k} = 2^{1-n}.
\end{equation}
In \cite{Wang}, Wang and Zhang obtained a more general result valid for $j\geq 1$ as follows
\begin{equation} \label{wang}
\sum_{k=0}^{\lfloor{n/2}\rfloor} \binom{n}{2k} \left(\frac{5}{4}\right)^{k}\! F_j^{2k}L_{j(n-2k)} E_{2k} = 2^{1-n} L_j^n.
\end{equation}
Castellanos \cite{Castellanos} found 
\begin{equation} \label{castellanos}
\sum_{k=0}^{n} \binom{2n}{2k}2^{-2k-1} L_{2(n-k)j}L_{j}^{2k} E_{2k} = \Big(\frac{5}{4}\Big)^n F_j^{2n}, 
\end{equation}
 which expresses even powers of Fibonacci numbers in terms of Lucas and Euler numbers.

Here, as usual, Fibonacci and Lucas numbers satisfy the recurrence $u_{n} = u_{n-1} + u_{n-2}$, $n\geq 2$, 
with initial conditions $F_0 = 0$, $F_1 = 1$ and $L_0 = 2$, $L_1 = 1$, respectively, 
whereas Euler numbers $(E_n)_{n\geq 0}$ are given by the power series
\begin{displaymath}
\sum_{n=0}^\infty E_n \frac{z^n}{n!} =\frac{1}{\cosh{z}}.
\end{displaymath} 

Fibonacci and Lucas numbers are entries \seqnum{A000045} and \seqnum{A000032} in the On-Line Encyclopedia of Integer Sequences \cite{Sloane}, respectively.

The Lucas-Euler pair may be regarded as the twin of the Fibonacci-Bernoulli pair. 
In the last years, there has been a growing interest in deriving new relations for these two pairs of sequences. 
Zhang and Ma \cite{Zhang} proved a relation between Fibonacci polynomials and Bernoulli numbers $(B_n)_{n\geq 0}$ defined by
\begin{displaymath}
\sum_{n=0}^\infty B_n \frac{z^n}{n!} = \frac{z}{e^z-1}.
\end{displaymath} 
The following identity is a special case of their result:
\begin{equation*} \label{zhama1}
\sum_{k=0}^{n} \binom{n}{k} 5^{\frac{n-k}{2}} F_{k} B_{n-k} =n\beta^{n-1},
\end{equation*}
where $\beta=(1-\sqrt{5})/2$, or, equivalently, 
\begin{equation} \label{zhama1}
\sum_{k=0}^{\lfloor n/2\rfloor} \binom{n}{2k} 5^k F_{n-2k} B_{2k} = \frac{nL_{n-1}}{2}.
\end{equation}
See also \cite{Ozdemir,Wang,Young,Zhang-Guo} for other results in this direction. 
Recently, Frontczak \cite{RF-JIS}, Frontczak and Goy \cite{RF-TG}, and Frontczak and Tomovski \cite{RF-Tomovski} proved
some generalizations of existing results. For instance, from \cite{RF-TG} we have
\begin{equation}\label{frogoy1}
\sum_{k=0}^n {n\choose k}  (\sqrt{5}F_j)^{n-k} F_{jk}B_{n-k} = n F_j \beta^{j(n-1)},
\end{equation}
which holds for all $j\geq 1$ and generalizes \eqref{zhama1} to an arithmetic progression, and
\begin{equation}\label{frogoy2}
\sum_{k=0}^{\lfloor n/2 \rfloor} {n\choose 2k}(20^k-5^k)F_{2j}^{2k}L_{2j(n-2k)}B_{2k} = \frac{5n}{2} F_{2j}F_{2j(n-1)}.
\end{equation}
Note, since $B_{2n+1}=0$ for $n\geq1$, from \eqref{frogoy1} we get  Kelinsky's formula \cite{Kelisky}
\begin{equation*}
\sum_{k=0}^{\lfloor{n}/{2}\rfloor} {n\choose 2k} 5^k F_j^{2k} F_{j(n-2k)}B_{2k} = \frac{n}{2} F_j L_{j(n-1)}.
\end{equation*}

In this paper, we present new identities linking Lucas numbers to Euler numbers (polynomials). 
The results stated are polynomial generalizations of \eqref{byrd1} and are complements of the recent discoveries 
from \cite{RF-JIS} and \cite{RF-TG}.

 Throughout the paper, we will work with different kind of polynomials of a complex variable $x$: 
Euler polynomials $(E_n(x))_{n\geq 0}$, Bernoulli polynomials $(B_n(x))_{n\geq 0}$, balancing polynomials $(B_n^*(x))_{n\geq 0}$, 
and Lucas-balancing polynomials $(C_n(x))_{n\geq 0}$. 

Euler and Bernoulli polynomials are famous mathematical objects and are fairly well understood. 
They are defined by \cite[Chapter 24]{Dilcher}
\begin{equation}\label{def_b}
H(x,z) = \sum_{n=0}^\infty B_n(x) \frac{z^n}{n!} = \frac{z e^{xz}}{e^z - 1} \qquad (|z|<2\pi) 
\end{equation} 
and
\begin{equation}\label{def_e}
I(x,z) = \sum_{n=0}^\infty E_n(x) \frac{z^n}{n!} = \frac{2e^{xz}}{e^z + 1} \qquad (|z|<\pi). 
\end{equation}
 
The numbers $B_n(0)=B_n$ are the famous Bernoulli numbers. Bernoulli numbers are rational numbers starting 
with $B_0=1, B_1=-1/2, B_2=1/6, B_4=-1/30$ and so on. Also, as already mentioned, $B_{2n+1}=0$ for $n\geq 1$. 
Euler numbers $E_n$ are obtained from $I(1/2,2z)$ that is 
\begin{equation}
E_n=2^n E_n(1/2).\label{En(1/2)}
\end{equation} 
In contrast to Bernoulli numbers, Euler numbers are integers where $E_0=1, E_2=-1, E_4=5$ and $E_{2n+1}=0$ for $n\geq 0$. 
Explicit formulas for the polynomials are
\begin{displaymath}
B_n(x) = \sum_{k=0}^n {n\choose k} B_{k}x^{n-k} \quad \mbox{and} \quad 
E_n(x) = \sum_{k=0}^n {n\choose k} \frac{E_{k}}{2^k}\Big ( x-\frac{1}{2}\Big )^{n-k}.
\end{displaymath}

Euler polynomials can be expressed in terms of Bernoulli polynomials via
\begin{displaymath}
E_{n}(x) = \frac{2}{n+1}\Big( B_{n+1} (x) - 2^{n+1} B_{n+1}\Big(\frac{x}{2}\Big)\Big).
\end{displaymath}
Particularly, 
\begin{equation}\label{E_n(0)}
E_n(0) = \frac{2\,(1-2^{n+1})}{n+1} B_{n+1}.
\end{equation}

Balancing polynomials are of younger age and are introduced in the next section.

\section{Balancing and Lucas-Balancing Polynomials}

Balancing polynomials $B_n^*(x)$ and Lucas-balancing polynomials $C_n(x)$ are generalizations of 
balancing and Lucas-balancing numbers \cite{RF-AMS}. These polynomials satisfy the recurrence
$w_n(x) = 6x w_{n-1}(x) - w_{n-2}(x)$, $\geq 2$, but with the respectively initial conditions 
$B_0^*(x)=0$, $B_1^*(x)=1$ and $C_0(x)=1$, $C_1(x)=3x$. The explicit formulas for these polynomials are
\begin{equation*}
B^*_n(x) = \frac{\lambda^n (x) - \lambda^{-n} (x)}{2\sqrt{9x^2-1}}
\quad \mbox{and} \quad 
C_n(x) = \frac{\lambda^n (x) + \lambda^{-n} (x)}{2}, 
\end{equation*}
where $\lambda (x)=3x + \sqrt{9x^2-1}$. Consult the papers \cite{RF-AMS,RF-TG-AMI,Kim,Meng} for more information about these polynomials. The numbers $B_n^*(1)=B_n^*$ and $C_n(1)=C_n$ are called balancing and Lucas-balancing numbers, respectively. These numbers are indexed in \cite{Sloane} under entries \seqnum{A001109} and \seqnum{A001541}.

Balancing and Lucas-balancing polynomials possess interesting properties. They are related to Chebyshev polynomials by simple scaling \cite[Lemma 2.1]{RF-AMS}. The exponential generating functions for balancing and Lucas-balancing polynomials are derived in \cite{RF-AMS,RF-TG-AMI}. Here, however, we will only need the results from \cite{RF-TG-AMI}: Let $b_1(x,z)$ and $b_2(x,z)$ be the exponential generating functions 
of odd and even indexed balancing polynomials, respectively. Then
\begin{align} \label{genf_b1}
 b_1(x,z) &= \sum_{n=0}^\infty B_{2n+1}^*(x) \frac{z^n}{n!} \nonumber \\
& = \frac{e^{(18x^2-1)z}}{\sqrt{9x^2-1}}\bigl( 3x \sinh (6x \sqrt{9x^2-1}z) + \sqrt{9x^2-1} \cosh (6x \sqrt{9x^2-1}z)\bigr)
\end{align}
and
\begin{equation} \label{genf_b2}
b_2(x,z) = \sum_{n=0}^\infty B_{2n}^*(x) \frac{z^n}{n!} = \frac{e^{(18x^2-1)z}}{\sqrt{9x^2-1}}\sinh (6x \sqrt{9x^2-1}z).
\end{equation}
Similarly, the exponential generating functions for Lucas-balancing polynomials are found to be
\begin{align} \label{genf_c1}
 c_1(x,z) &= \sum_{n=0}^\infty C_{2n+1}(x) \frac{z^n}{n!} \nonumber \\
& = e^{(18x^2-1)z} \bigl ( 3x \cosh (6x \sqrt{9x^2-1}z) + \sqrt{9x^2-1} \sinh (6x \sqrt{9x^2-1}z)\bigr)
\end{align}
and
\begin{equation} \label{genf_c2}
c_2(x,z) = \sum_{n=0}^\infty C_{2n}(x) \frac{z^n}{n!} = e^{(18x^2-1)z}\cosh (6x \sqrt{9x^2-1}z).
\end{equation}

Connections between Bernoulli polynomials $B_n(x)$ and balancing polynomials $B_n^*(x)$ 
have been established in the recent papers \cite{RF-JIS,RF-TG}. They are interesting, 
as they instantly give relations between Bernoulli numbers and Fibonacci and Lucas numbers. 
The links are the following evaluations \cite{RF-AMS}
\begin{equation}\label{link1}
B_n^*\Big (\omega_s\frac{L_{s}}{6} \Big ) = \omega_s^{n-1}\frac{F_{sn}}{F_{s}}\,, \qquad C_n\Big (\omega_s\frac{L_{s}}{6} \Big  ) = \omega_s^{n}\frac{L_{sn}}{2},
\end{equation}
where $\omega_s=1$, if $s$ is even, and $\omega_s=i=\sqrt{-1}$, if $s$ is odd.
These links will be used to prove our results. 

\section{Relations between Euler and Balancing \\(Lucas-Balancing) Polynomials}

We start with the following result involving even indexed balancing and  Lucas-balancing polynomials.
\begin{theorem}\label{thm1}
For each $n\geq 1$ and $x\in\mathbb{C}$, we have
\begin{equation}\label{maineq1}
\sum_{k=1}^{\lfloor {n}/{2} \rfloor} {n-1\choose 2k-1} C_{2(n-2k)}(x)\big(144x^2(9x^2-1)\big)^{k}E_{2k-1}(0) = 12x(1-9x^2)B_{2n-2}^*(x).
\end{equation}
\end{theorem}
\begin{proof}
Since $\tanh{z} = 1 - \frac{2}{e^{2z}+1}$, from \eqref{def_e} we get
\begin{equation*}
I(0,12x\sqrt{9x^2-1}z) = 1 - \tanh(6x\sqrt{9x^2-1}z)
\end{equation*}
and, by \eqref{genf_b2} and \eqref{genf_c2}, 
\begin{align*}
\sum_{n=0}^\infty \Big (\sum_{k=0}^{n-1} {n\choose k}& C_{2k}\bigl(12x\sqrt{9x^2-1})^{n-k} E_{n-k}(0)+C_{2n}(x) \Big)\frac{z^n}{n!} \\
&=  c_2(x,z) I(0,12x\sqrt{9x^2-1}z) \\
& = e^{(18x^2-1)z} \bigl(\cosh (6x \sqrt{9x^2-1}z) - \sinh (6x \sqrt{9x^2-1}z)\bigr) \\
& = c_2(x,z) - \sqrt{9x^2-1} b_2(x,z) \\
& = \sum_{n=0}^\infty \bigl(C_{2n}(x) - \sqrt{9x^2-1} B_{2n}^*(x)\bigr) \frac{z^n}{n!}.
\end{align*}

Thus, 
\begin{gather*}
\sum_{k=0}^n {n\choose k} C_{2(n-k)}\bigl(12x\sqrt{9x^2-1})^{k} E_{k}(0)= C_{2n}(x) - \sqrt{9x^2-1} B_{2n}^*(x).
\end{gather*}
Since $E_{2n-1}=0$ for $n\geq1$, after some algebra we have \eqref{maineq1}.
\end{proof}
\begin{corollary}\label{cor1}
For each $n\geq 1$ and $j\geq 1$,
\begin{equation}\label{eqcor1}
\sum_{k=0}^{\lfloor {n}/{2} \rfloor} {n-1\choose 2k-1}5^{k-1}F_{2j}^{2k-1}L_{2j(n-2k)} E_{2k-1}(0) = - F_{2j(n-1)}.
\end{equation}
%
\end{corollary}
\begin{proof} 
Evaluate \eqref{maineq1} at the $x=\omega_jL_{j}/6$ and use the links from \eqref{link1}. 
To simplify recall that $L_n^2-5F_n^2=(-1)^n4$ and $F_{2n} = F_nL_n$.
\end{proof}

Using \eqref{E_n(0)}, we can write \eqref{eqcor1} as
\begin{equation*}
\sum_{k=0}^{\lfloor {n}/{2} \rfloor} {n-1\choose 2k-1}\frac{20^k-5^k}{k} F_{2j}^{2k-1} L_{2j(n-2k)} B_{2k} = 5F_{2j(n-1)},
\end{equation*}
which is easily reduced to \eqref{frogoy2}.

We also have the following interesting identity.
\begin{theorem} \label{thm2}
For each $n\geq 0$ and $x\in\mathbb{C}$, we have the relation
\begin{equation} \label{maineq2}
\sum_{k=0}^{\lfloor{n}/{2}\rfloor} {n\choose 2k} C_{2(n-2k)}(x) \big(36x^2({9x^2-1})\big)^{k} E_{2k} = \bigl(18x^2 - 1 \bigr)^{n}.
\end{equation}
\end{theorem}
\begin{proof}
The result is a consequence of the fact that
\begin{displaymath}
c_2(x,z) I(1/2,12x\sqrt{9x^2-1}z) = e^{(18x^2-1)z}.
\end{displaymath}
\end{proof}
\begin{corollary}\label{cor1-1}
	For each $n\geq 0$ and $j\geq 1$,
\begin{equation}\label{interest}
\sum_{k=0}^{\lfloor n/2\rfloor}{n\choose 2k} \Bigl(\frac{5}{4}\Bigr)^{k}F^{2k}_{2j}L_{2j(n-2k)}E_{2k} = 2^{1-n} L_{2j}^n.
\end{equation}
\end{corollary}
\begin{proof} 
Evaluate \eqref{maineq2} at the point $x=\omega_jL_{j}/6$ and use the links from \eqref{link1}. 
When simplifying you will also need the formula $L_n^2-L_{2n}=(-1)^{n}2$.
\end{proof}

Interestingly, if $j=1/2$ from \eqref{interest} we obtain Byrd's result \eqref{byrd1}. 
Also, when $j=1$ and $j=2$, from \eqref{interest} we obtain the following Lucas-Euler relations:
\begin{equation*}
\sum_{k=0}^{\lfloor n/2 \rfloor} {n\choose 2k} \Big(\frac{5}{4} \Big )^{k} L_{2(n-2k)} E_{2k} = 2\Big (\frac{3}{2}\Big )^n
\end{equation*}
and 
\begin{equation}\label{idxx}
\sum_{k=0}^{\lfloor n/2 \rfloor} {n\choose 2k}  \Big ( \frac{45}{4} \Big )^{k} L_{4(n-2k)} E_{2k} = 2\Big (\frac{7}{2}\Big )^n,
\end{equation}
respectively. The first example appears as equation (31) in \cite{RF-JIS}.

A different expression for the sum on the left of \eqref{maineq2} is stated next.
\begin{theorem} \label{thm3}
For each $n\geq 0$ and $x\in\mathbb{C}$, we have 
\begin{align} \label{maineq3}
\sum_{k=0}^{\lfloor n/2\rfloor} {n\choose 2k} C_{2(n-2k)}(x) \big(36x^2(9x^2-1)\big)^{k} E_{2k}&\nonumber\\ 
= \sum_{k=0}^n {n\choose k} \big ( C_{2k}(x) - &\sqrt{9x^2-1}B_{2k}^*(x)\big) (6x\sqrt{9x^2-1})^{n-k}.
\end{align}
\end{theorem}
\begin{proof}
We use the identity $I(1/2,2z) = e^{z}(1 - \tanh z)$, from which the functional equation follows
\begin{displaymath}
c_2(x,z) I(1/2,12x\sqrt{9x^2-1}z) = e^{(6x\sqrt{9x^2-1})z}\big (c_2(x,z) - \sqrt{9x^2-1}b_2(x,z)\big ).
\end{displaymath}

Thus,
\begin{align*}
\sum_{k=0}^n {n\choose k} C_{2k}(x) (6x\sqrt{9x^2-1})^{n-k} E_{n-k}&\\ 
= \sum_{k=0}^n {n\choose k} &\big ( C_{2k}(x) - \sqrt{9x^2-1}B_{2k}^*(x)\big) (6x\sqrt{9x^2-1})^{n-k},
\end{align*}
that is equivalent to \eqref{maineq3}.
\end{proof}

\begin{theorem} \label{thm4}
For each $n\geq 0$ and $x\in\mathbb{C}$, it is true that 
\begin{equation} \label{maineq4}
\sum_{k=0}^n {n\choose k} C_{2(n-k)}(x) (12x\sqrt{9x^2-1})^{k} E_{k}(x) = \big ( 18x^2 - 1 + 6x(2x-1)\sqrt{9x^2-1}\big )^{n}.
\end{equation}
\end{theorem}
\begin{proof}
The functional relation ${\cosh(z/2)} I(x,z) = e^{(x-1/2)z}$ produces immediately
\begin{displaymath}
c_2(x,z)I(x,12x\sqrt{9x^2-1}z) = e^{(18x^2 - 1 + 6x(2x-1)\sqrt{9x^2-1})z}.
\end{displaymath}
Comparing the coefficients of $z$ in the power series expansions on both sides gives the identity.
\end{proof}

When $x=1/2$, then we recover \eqref{idxx}, by \eqref{En(1/2)}. 

\section{Other Special Polynomial Identities}

The following result appears as Theorem 13 in \cite{RF-TG}:
For each $n\geq 0$, $j\geq 1$, and $x\in\mathbb{C}$, we have 
\begin{equation*}\label{FBpol1}
\sum_{k=0}^n {n\choose k} F_{jk} (\sqrt{5}F_j)^{n-k} B_{n-k}(x) = n F_j \big ((\sqrt{5}x+\beta)F_j + F_{j-1}\big )^{n-1}
\end{equation*}
and
\begin{equation*}\label{FBpol2}
\sum_{k=0}^n {n\choose k} F_{jk} (-\sqrt{5}F_j)^{n-k} B_{n-k}(x) = n F_j \big ((\alpha - \sqrt{5}x)F_j + F_{j-1}\big )^{n-1},
\end{equation*}
where $\alpha=(1+\sqrt{5})/2$ is the golden ratio and $\beta=(1-\sqrt{5})/2=-1/\alpha$. 

Now, we present the analogue result for the Lucas-Euler pair:
\begin{theorem} \label{thm5} 
The following polynomial identity is valid for all $n\geq 0$, $j\geq 1$, and $x\in\mathbb{C}$:
\begin{equation}\label{main51}
\sum_{k=0}^n {n\choose k} L_{jk} (\sqrt{5}F_j)^{n-k} E_{n-k}(x) = 2 \big ((\sqrt{5}x+\beta)F_j + F_{j-1}\big )^{n}
\end{equation}
and
\begin{equation}\label{main52}
\sum_{k=0}^n {n\choose k} L_{jk} (-\sqrt{5}F_j)^{n-k} E_{n-k}(x) = 2 \big ((\alpha - \sqrt{5}x)F_j + F_{j-1}\big )^{n}.
\end{equation}
\end{theorem}
\begin{proof}
Let $L(z)$ be the exponential generating function for $(L_{jn})_{n\geq 0},j\geq 1$. Then, using the Binet formula for $L_{n}$ we get
\begin{equation*}
L(z) = 2 e^{(1/2 F_j + F_{j-1})z}\cosh\Big ( \frac{\sqrt{5}F_j}{2}  z\Big ).
\end{equation*}
Thus, it follows that
\begin{align*}
\sum_{n=0}^\infty \Big (\sum_{k=0}^n {n\choose k} L_{jk} (\sqrt{5}F_j)^{n-k} E_{n-k}(x) \Big ) \frac{z^n}{n!} & =
L(z) I(x,\sqrt{5}F_j z) \\
& = 2 \, e^{((x-1/2)\sqrt{5}F_j + 1/2 F_j + F_{j-1})z} \\
& = 2 \, e^{((\sqrt{5}x + \beta)F_j + F_{j-1})z}.
\end{align*}
This proves the first equation. The second follows upon replacing $x$ by $1-x$
and using $E_n(1-x)=(-1)^nE_n(x)$ and $\alpha-\beta=\sqrt{5}$.
\end{proof}

Note that the relations \eqref{main51} and \eqref{main52} provide a generalization of \eqref{interest}. 
To see this, notice that they can be written more compactly as
\begin{equation}\label{main53}
\sum_{k=0}^n {n\choose k} L_{jk} (\pm\sqrt{5}F_j)^{n-k} E_{n-k}(x) = 2^{1-n} \big (L_j\pm \sqrt{5}F_j(2x-1)\big)^{n}.
\end{equation}
Now, if $x=1/2$, we get 
\begin{equation*}
\sum_{k=0}^n {n\choose k}\big(\pm\sqrt{5}F_{j}\big)^{n-k} 2^k L_{jk} E_{n-k} =2L_j^n,
\end{equation*}
which is equivalent to \eqref{interest}. We also mention the nice and curious identities
\begin{equation*}
\sum_{k=0}^n {n\choose k}\big(\pm\sqrt{5}F_{j}\big)^{n-k} L_{jk} E_{n-k}(\alpha) =2 (\pm 1)^n L_{j\pm 1}^n,
\end{equation*}
and
\begin{equation*}
\sum_{k=0}^n {n\choose k}\big(\pm\sqrt{5}F_{j}\big)^{n-k} L_{jk} E_{n-k}(\beta) =2 (\mp 1)^n L_{j\mp 1}^n,
\end{equation*}
which can be deduced from \eqref{main53} and $5F_n=L_{n+1}+L_{n-1}$.

We conclude this presentation with the following interesting corollary.

\begin{corollary}\label{cor_main5}
Let $n$, $j$ and $q$ be integers with $n$, $j\geq 1$ and $q$ odd. Then it holds that
\begin{equation*} \label{eqcor_main5}
\sum_{k=0}^{n} {n\choose k}(\sqrt{5}F_j )^{n-k} \big ( q^{-(n-k)} - 1 \big ) L_{jk} E_{n-k}(0) =
2 q^{-n} \sum_{r=1}^{q-1} (-1)^r \big ( r\alpha^j + (q - r)\beta^j \big )^{n}.
\end{equation*}
\end{corollary}
\begin{proof}  
The known multiplication formula for Euler polynomials for odd $q$ \cite[Chapter 24]{Dilcher}
\begin{equation*}
q^{n} \sum_{r=0}^{q-1} (-1)^r E_n \Big (x + \frac{r}{q}\Big ) = E_n(qx)
\end{equation*}
yields
\begin{equation*}
\sum_{r=1}^{q-1} (-1)^r E_n \Big (\frac{r}{q}\Big ) = \big ( q^{-n} - 1 \big ) E_{n}(0).
\end{equation*}

Therefore, 
\begin{align*}
\sum_{k=0}^{n} {n\choose k} L_{jk} (\sqrt{5}F_j)^{n-k} \big ( q^{-(n-k)} - 1 \big ) E_{n-k}(0) & = 2 \sum_{r=1}^{q-1} (-1)^r
\Big ( \Big ( \sqrt{5}\,\frac{r}{q}+\beta\Big ) F_j + F_{j-1}\Big )^{n} \\
& = 2 q^{-n} \sum_{r=1}^{q-1} (-1)^r \big ( \sqrt{5}r F_j + q(\beta F_j + F_{j-1})\big )^{n} \\
& = 2 q^{-n} \sum_{r=1}^{q-1} (-1)^r \big ( r\alpha^j + (q - r)\beta^j \big )^{n}.
\end{align*}
\end{proof}

The special instances for $j=1$, and $q=3$ and $q=5$, respectively, take the form
\begin{equation*}
\sum_{k=0}^{n} {n\choose k} L_{k} (\sqrt{5})^{n-k} \big ( 3^{-(n-k)} - 1 \big ) E_{n-k}(0)
= 2\cdot 3^{-n} \sqrt{5} F_{2n}
\end{equation*}
and
\begin{equation*}
\sum_{k=0}^{n} {n\choose k} L_{k} (\sqrt{5})^{n-k} \big ( 5^{-(n-k)} - 1 \big ) E_{n-k}(0)
= \left.
  \begin{cases}
    2\cdot 5^{(1-n)/2} (F_{2n} - F_{n}), & \mbox{if $n$ is even}; \\
    2\cdot 5^{-n/2} (L_{2n} - L_{n}), & \mbox{if $n$ is odd}.
  \end{cases}
  \right.
\end{equation*}

\section{Conclusion}
In this paper, we have documented identities relating Euler numbers (polynomials) to
balancing and Lucas-balancing polynomials. We have also derived a general identity involving
Euler polynomials and Lucas numbers in arithmetic progression. 
All results must be seen as companion results to the 
Fibonacci-Bernoulli pair from \cite{RF-TG}. In the future, we will work
on more identities connecting Bernoulli/Euler numbers (polynomials) with Fibonacci/Lucas numbers (polynomials).

%

\end{document}